\documentclass[10pt]{article}

\usepackage{times}
\usepackage{array, amssymb, amsmath, graphics}
\newtheorem{theorem}{Theorem}
\newtheorem{example}{Example}

\newtheorem{lemma}{Lemma}

\newenvironment{proof}{{\bf Proof.}}{\hspace*{1mm}\hfill\rule{2mm}{2mm}}

\newtheorem{pretheorema}{{\bf Theorem}}

\def\a{\underline{\bf 1}}
\def\s{{$\star$}}
\def\z{{$\cdot$}}
\def\M#1#2#3{\mbox{\rm M}(#1\mbox{-}(#2,#3))}
\def\T#1#2#3{#1\mbox{-}(#2,#3)\mbox{ \rm Latin trade}}
\def\SFT#1#2#3{#1\mbox{-}(#2,#3)\mbox{ \sf Latin trade}}

\def\I#1#2#3{#1\mbox{-}(#2,#3)\mbox{ \rm intercalate}}
\def\SFI#1#2#3{#1\mbox{-}(#2,#3)\mbox{ \sf intercalate}}
\def\ITI#1#2#3{#1\mbox{-}(#2,#3)\mbox{ \it intercalate}}
\def\m#1#2{\raise 0.2ex\hbox{
    ${#1_{\displaystyle #2}}$}}
\def\xx#1{\raise 0.5ex\hbox{
    ${#1}$}}
\def\n#1{{$#1$}}
\def\arraystretch{1.0}                
\def\pls{partial Latin square}
\def\x{{\bf x}}
\def\y{{\bf y}}
\def\u{{\bf u}}
\title{\bf A linear algebraic approach to orthogonal arrays and Latin squares}
\author{A. A. Khanban\footnote{Department of Computing, Imperial
College London, London SW7 2BZ, United Kingdom. ({\tt
khanban@doc.ic.ac.uk})} , M. Mahdian\footnote{Yahoo! Research,
Santa Clara, CA, USA. ({\tt mahdian@alum.mit.edu})} , and E. S.
Mahmoodian\footnote{Department of Mathematics, Sharif University
of Technology, and Institute for Studies in Theoretical Physics
and Mathematics (IPM), Tehran, Iran. ({\tt emahmood@sharif.edu})}}
\date{}
\begin{document}

\maketitle


\begin{abstract}
To study orthogonal arrays and signed orthogonal arrays,
Ray-Chaudhuri and Singhi (1988 and 1994) considered some module
spaces. Here, using a linear algebraic approach we define an
inclusion matrix and find its rank. In the special case of Latin
squares we show that there is a straightforward algorithm for
generating a basis for this matrix using the so-called
intercalates. We also extend this last idea.
\end{abstract}


\section{Introduction and preliminaries}

To show the existence of signed orthogonal arrays, Ray-Chaudhuri
and Singhi considered a space of linear forms in variables and
calculated its rank, see \cite{MR924450}. Later they pointed out
an error in their calculation and provided a correction,
see~\cite{MR1275740}. Here we define a natural inclusion matrix
corresponding to orthogonal arrays and signed orthogonal arrays.
We compute the rank of this matrix and study  bases of its null
space. This provides helpful insight for studying these objects.
In the special case of Latin squares we show that there is a
straightforward algorithm for generating a basis for this matrix
using the so-called intercalates. We also extend this  idea for
more general cases.

We follow the notations of \cite{MR924450} as much as possible.
Let $V := \{0, 1, \ldots, v-1\}$ and $V^k$ be the set of all
ordered $k$-tuples of the elements of $V$, i.e., $V^k := \{(x_1,
\ldots, x_k) \mid x_i\in V, i=1, \ldots, k\}$. Also, let $V_I^t
:=\{(u_1,\ldots,u_t)_I\mid u_i\in V, i=1, \ldots, t\}$, where $I$
is a subset of size $t$ of the set $\{1,\ldots,k\}$. For a pair of
elements of $V^k$ and $V_I^t$, where $I= \{i_1,\ldots,i_t\}$ and
$i_1 < \cdots < i_t$, we define:
\ \ $(u_1,\ldots,u_t)_I\in (x_1,\ldots, x_k) \quad
\Longleftrightarrow \quad u_j=x_{i_j},\qquad j=1,\ldots,t. $

An {\sf orthogonal array} ${\rm OA}_t(v,k,\lambda)$ on a set $V$
is a collection of $k$-tuples of elements of $V$ such that for
each $I \subseteq \{1,\ldots,k\}$, $|I|=t$, every element of
$V_I^t$ belongs to exactly $\lambda$ elements of the collection.
Orthogonal arrays were first defined by Rao~\cite{MR0022821} and
have been used in studying designs and codes.

Next we define the {\sf $t$-inclusion matrix} $\M{t}{v}{k}$.
Columns of this matrix correspond to the elements of $V^k$ (in
lexicographic order) and its rows correspond to the elements of
$\cup_I V_I^t$, where the union is over all $t$-subsets of
$\{1,\ldots,k\}$. The entries of the matrix are 0 or 1, and are
defined as follows. \
\ \ $M_{(u_1,\ldots,u_t)_I, (x_1,\ldots,x_k)} = 1 \quad
\Longleftrightarrow \quad (u_1,\ldots,u_t)_I\in (x_1,\ldots, x_k).
$
%
\begin{example}\label{ex:m-2-(3,3)}
In Figure~\ref{fig:m-2-(3,3)}, the matrix $\M{2}{3}{3}$ is shown.
To make it more readable, 0 entries of the matrix are represented
by ``\z'' signs. The ``\s'' signs may be ignored for time being,
they will be explained in Section~\ref{sec:OA}.
\end{example}
%
\begin{figure}[ht]\label{fig:m-2-(3,3)}
\begin{center}
{\footnotesize
\def\arraystretch{1.0}
\begin{tabular}{r@{\hspace{0.3mm}}r@{\hspace{0mm}}|c@{\hspace{1.9mm}}c@{\hspace{1.9mm}}c@
{\hspace{1.9mm}}c@{\hspace{1.9mm}}c@{\hspace{1.9mm}}c@{\hspace{1.9mm}}c@{\hspace{1.9mm}}c@
{\hspace{1.9mm}}c@{\hspace{1.9mm}}c@{\hspace{1.9mm}}c@{\hspace{1.9mm}}c@{\hspace{1.9mm}}c@
{\hspace{1.9mm}}c@{\hspace{1.9mm}}c@{\hspace{1.9mm}}c@{\hspace{1.9mm}}c@{\hspace{1.9mm}}c@
{\hspace{1.9mm}}c@{\hspace{1.9mm}}c@{\hspace{1.9mm}}c@{\hspace{1.9mm}}c@{\hspace{1.9mm}}c@
{\hspace{1.9mm}}c@{\hspace{1.9mm}}c@{\hspace{1.9mm}}c@{\hspace{1.9mm}}c@{\hspace{1.9mm}}c}
\multicolumn{2}{r}{}&\s&\s&\s&\s&\s&\s&\s&\s&\s&\s&\s&\s&\s& &
&\s& & &\s&\s&\s&\s& & &\s& &
\\[-0.8mm]
\multicolumn{2}{r}{}&0&0&0&0&0&0&0&0&0&1&1&1&1&1&1&1&1&1&2&2&2&2&2&2&2&2&2&\hspace*{1cm}\\[-0.8mm]
\multicolumn{2}{r}{}&0&0&0&1&1&1&2&2&2&0&0&0&1&1&1&2&2&2&0&0&0&1&1&1&2&2&2\\[-0.8mm]
\multicolumn{2}{r}{}&0&1&2&0&1&2&0&1&2&0&1&2&0&1&2&0&1&2&0&1&2&0&1&2&0&1&2\\[-0.8mm]
\cline{3-29}
&$(0,0)_{\{1,2\}}$&1&1&1&\z&\z&\z&\z&\z&\z&\z&\z&\z&\z&\z&\z&\z&\z&\z&\z&\z&\z&\z&\z&\z&\z&\z&
     \z{\vrule height4mm width0em depth-0.1mm}\\[-0.8mm]
&$(0,1)_{\{1,2\}}$&\z&\z&\z&1&1&1&\z&\z&\z&\z&\z&\z&\z&\z&\z&\z&\z&\z&\z&\z&\z&\z&\z&\z&\z&\z&\z\\[-0.8mm]
&$(0,2)_{\{1,2\}}$&\z&\z&\z&\z&\z&\z&1&1&1&\z&\z&\z&\z&\z&\z&\z&\z&\z&\z&\z&\z&\z&\z&\z&\z&\z&\z\\[-0.8mm]
&$(1,0)_{\{1,2\}}$&\z&\z&\z&\z&\z&\z&\z&\z&\z&1&1&1&\z&\z&\z&\z&\z&\z&\z&\z&\z&\z&\z&\z&\z&\z&\z\\[-0.8mm]
\s&$(1,1)_{\{1,2\}}$&\z&\z&\z&\z&\z&\z&\z&\z&\z&\z&\z&\z&\a&1&1&\z&\z&\z&\z&\z&\z&\z&\z&\z&\z&\z&\z\\[-0.8mm]
\s&$(1,2)_{\{1,2\}}$&\z&\z&\z&\z&\z&\z&\z&\z&\z&\z&\z&\z&\z&\z&\z&\a&1&1&\z&\z&\z&\z&\z&\z&\z&\z&\z\\[-0.8mm]
&$(2,0)_{\{1,2\}}$&\z&\z&\z&\z&\z&\z&\z&\z&\z&\z&\z&\z&\z&\z&\z&\z&\z&\z&1&1&1&\z&\z&\z&\z&\z&\z\\[-0.8mm]
\s&$(2,1)_{\{1,2\}}$&\z&\z&\z&\z&\z&\z&\z&\z&\z&\z&\z&\z&\z&\z&\z&\z&\z&\z&\z&\z&\z&\a&1&1&\z&\z&\z\\[-0.8mm]
\s&$(2,2)_{\{1,2\}}$&\z&\z&\z&\z&\z&\z&\z&\z&\z&\z&\z&\z&\z&\z&\z&\z&\z&\z&\z&\z&\z&\z&\z&\z&\a&1&1\\[1.5mm]
&$(0,0)_{\{1,3\}}$&1&\z&\z&1&\z&\z&1&\z&\z&\z&\z&\z&\z&\z&\z&\z&\z&\z&\z&\z&\z&\z&\z&\z&\z&\z&\z\\[-0.8mm]
&$(0,1)_{\{1,3\}}$&\z&1&\z&\z&1&\z&\z&1&\z&\z&\z&\z&\z&\z&\z&\z&\z&\z&\z&\z&\z&\z&\z&\z&\z&\z&\z\\[-0.8mm]
&$(0,2)_{\{1,3\}}$&\z&\z&1&\z&\z&1&\z&\z&1&\z&\z&\z&\z&\z&\z&\z&\z&\z&\z&\z&\z&\z&\z&\z&\z&\z&\z\\[-0.8mm]
\s&$(1,0)_{\{1,3\}}$&\z&\z&\z&\z&\z&\z&\z&\z&\z&\a&\z&\z&1&\z&\z&1&\z&\z&\z&\z&\z&\z&\z&\z&\z&\z&\z\\[-0.8mm]
\s&$(1,1)_{\{1,3\}}$&\z&\z&\z&\z&\z&\z&\z&\z&\z&\z&\a&\z&\z&1&\z&\z&1&\z&\z&\z&\z&\z&\z&\z&\z&\z&\z\\[-0.8mm]
\s&$(1,2)_{\{1,3\}}$&\z&\z&\z&\z&\z&\z&\z&\z&\z&\z&\z&\a&\z&\z&1&\z&\z&1&\z&\z&\z&\z&\z&\z&\z&\z&\z\\[-0.8mm]
\s&$(2,0)_{\{1,3\}}$&\z&\z&\z&\z&\z&\z&\z&\z&\z&\z&\z&\z&\z&\z&\z&\z&\z&\z&\a&\z&\z&1&\z&\z&1&\z&\z\\[-0.8mm]
\s&$(2,1)_{\{1,3\}}$&\z&\z&\z&\z&\z&\z&\z&\z&\z&\z&\z&\z&\z&\z&\z&\z&\z&\z&\z&\a&\z&\z&1&\z&\z&1&\z\\[-0.8mm]
\s&$(2,2)_{\{1,3\}}$&\z&\z&\z&\z&\z&\z&\z&\z&\z&\z&\z&\z&\z&\z&\z&\z&\z&\z&\z&\z&\a&\z&\z&1&\z&\z&1\\[1.8mm]
\s&$(0,0)_{\{2,3\}}$&\a&\z&\z&\z&\z&\z&\z&\z&\z&1&\z&\z&\z&\z&\z&\z&\z&\z&1&\z&\z&\z&\z&\z&\z&\z&\z\\[-0.8mm]
\s&$(0,1)_{\{2,3\}}$&\z&\a&\z&\z&\z&\z&\z&\z&\z&\z&1&\z&\z&\z&\z&\z&\z&\z&\z&1&\z&\z&\z&\z&\z&\z&\z\\[-0.8mm]
\s&$(0,2)_{\{2,3\}}$&\z&\z&\a&\z&\z&\z&\z&\z&\z&\z&\z&1&\z&\z&\z&\z&\z&\z&\z&\z&1&\z&\z&\z&\z&\z&\z\\[-0.8mm]
\s&$(1,0)_{\{2,3\}}$&\z&\z&\z&\a&\z&\z&\z&\z&\z&\z&\z&\z&1&\z&\z&\z&\z&\z&\z&\z&\z&1&\z&\z&\z&\z&\z\\[-0.8mm]
\s&$(1,1)_{\{2,3\}}$&\z&\z&\z&\z&\a&\z&\z&\z&\z&\z&\z&\z&\z&1&\z&\z&\z&\z&\z&\z&\z&\z&1&\z&\z&\z&\z\\[-0.8mm]
\s&$(1,2)_{\{2,3\}}$&\z&\z&\z&\z&\z&\a&\z&\z&\z&\z&\z&\z&\z&\z&1&\z&\z&\z&\z&\z&\z&\z&\z&1&\z&\z&\z\\[-0.8mm]
\s&$(2,0)_{\{2,3\}}$&\z&\z&\z&\z&\z&\z&\a&\z&\z&\z&\z&\z&\z&\z&\z&1&\z&\z&\z&\z&\z&\z&\z&\z&1&\z&\z\\[-0.8mm]
\s&$(2,1)_{\{2,3\}}$&\z&\z&\z&\z&\z&\z&\z&\a&\z&\z&\z&\z&\z&\z&\z&\z&1&\z&\z&\z&\z&\z&\z&\z&\z&1&\z\\[-0.8mm]
\s&$(2,2)_{\{2,3\}}$&\z&\z&\z&\z&\z&\z&\z&\z&\a&\z&\z&\z&\z&\z&\z&\z&\z&1&\z&\z&\z&\z&\z&\z&\z&\z&1\\[-0.8mm]
\end{tabular}
}
\end{center}
\caption{$\M{2}{3}{3}$}
\end{figure}

A {\sf Latin square} $L$ of order $v$ is a $v\times v$ array with
entries chosen from a set, say $V=\{0,1,\dots, v-1\}$ in such a
way that each element of $V$ occurs precisely once in each row
and in each column of the array. For ease of exposition, a Latin
square $L$ will be represented by a set of ordered triples
$\{(i,j;L_{ij})\mid {\rm element\ } L_{ij} {\rm \ occurs \ in\ }
{\rm cell\ } (i,j) {\rm \ of\ the\ array}\}$.

It is folkloric that any Latin square of order $v$ is equivalent
to an ${\rm OA}_2(v,3,1)$. It is easy to see that any ${\rm
OA}_t(v,k,\lambda)$ can be thought of a solution to the equation
\begin{equation}\label{eq:eqnf}
{\rm M}{\bf F}= \lambda \overline{\bf 1},
\end{equation}
where ${\rm M}=\M{t}{v}{k}$, $\overline{{\bf 1}}$ is a vector of
appropriate size with all components equal to 1, and ${\bf F}$ is
a non-negative integer-valued frequency vector, i.e., ${\bf
F}(\x)$ represents the number of times that OA contains the
ordered $k$-tuple $\x$. Our discussion will be based on the field
of real numbers.

In Section~\ref{sec:OA} we find the nullity and the rank of
$\M{t}{v}{k}$.  In Section~\ref{sec:Latinsquare} we show that a
very simple basis exists in the case of $\M{2}{v}{3}$, i.e., when
the OA corresponds to a Latin square, which consist of so-called
intercalates. In Section~\ref{sec:nullspace} we generalize the
result of Section~\ref{sec:Latinsquare} for any $\M{t}{v}{t+1}$.

\section{Orthogonal arrays}
\label{sec:OA}

The main result of this section is the following theorem,

\begin{theorem}\label{th:rank}
The rank of the matrix $\M{t}{v}{k}$ is equal to
\[
{\rm rank}(M) = \sum_{i=0}^t{k\choose i}(v-1)^i.
\]
\end{theorem}

This theorem results from the following lemmas. But first we need
the following notations. For every ordered $k$-tuple
$\x=(x_1,\ldots,x_k)$, the set $F_{\x}$ is defined as
\[
F_\x = \{(z_1,\ldots,z_k)\mid z_i\in\{0,x_i\}, i=1,\ldots,k\}.
\]
Also, we define $A_\x = \{i\mid x_i\neq 0\}$, $L_\x = |A_\x|$,
and let $C_\x$ denote the column of the matrix $M$ corresponding
to the $k$-tuple $\x$.

\begin{lemma}\label{lem:16}
For every $\x\in F_\y$, $\x\neq\y$, we have $L_\x<L_\y$ and
$\x\prec \y$, where $\prec$ denotes the lexicographic order.
\end{lemma}

\begin{proof}
Clearly, for every $i$, if $x_i\neq 0$, then $y_i=x_i\neq 0$.
Therefore, $L_\x\leq L_\y$. Now if the equality $L_\x=L_\y$ holds,
then for every non-zero $y_i$, the corresponding $x_i$ is non-zero
and hence equal to $y_i$. This implies $\x=\y$, contradicting the
hypothesis. For the second part, notice that for every $i$, either
$x_i=y_i$, or \, $0=x_i\leq y_i$. Hence $\x\preceq \y$ and since
$\x\neq \y$, we have $\x\prec\y$.
\end{proof}

\begin{lemma}\label{lem:17}
The number of linearly independent rows of $\M{t}{v}{k}$ is at
least the number of columns $C_\x$ with $L_\x\leq t$.
\end{lemma}

\begin{proof}
We show that for every column $C_\x$ with $L_\x\leq t$, there
exists a corresponding row such that its pivot 1 (that is the
first 1 in that row) is in column $C_\x$. These rows are clearly
linearly independent.
Corresponding to the column $C_\x$, we construct an element
$\u_I=(u_1,\ldots,u_t)_I$ of $V_I^t$ as follows.
Since $L_\x\leq t$, we can pick a set $I=\{i_1,\ldots,i_t\}$ with
$A_\x\subseteq I$. Note that there may be more than one choice for
$I$, but choosing any of them will serve our purpose. For every
$j=1,\ldots, t$, we let $u_j = x_{i_j}$.
This defines an element $\u_I$ such that $\u_I\in\x$. Therefore,
there is a 1 in the intersection of the column $C_\x$, and the
row corresponding to $\u_I$ in $\M{t}{v}{k}$.

Next, we prove that every $\y$ such that $\u_I\in\y$ satisfies
$\x\preceq \y$. This would imply that the first 1 in the row
corresponding to $\u_I$ lies in the column $C_\x$, which completes
the proof. Since for every $j=1,\ldots,t$, $u_j=x_{i_j}$ and $u_j
= y_{i_j}$, we have $x_i=y_i$ for every $i\in I$. Also, since
$x_i=0$ for every $i\not\in I$, we have $x_i\in\{0,y_i\}$ for
every $i$, or in other words, $\x\in F_\y$. Therefore, by
Lemma~\ref{lem:16}, $\x\preceq \y$.
\end{proof}

\begin{lemma}\label{lem:19}
${\rm rank}(\M{t}{v}{k})\geq\sum_{i=0}^t{k\choose i} (v-1)^i$.
\end{lemma}

\begin{proof}
By Lemma \ref{lem:17}, the rank of $M$ is at least the number of
columns $C_\x$ with $L_\x\leq t$. It is easy to note that the
number of vectors $\x$ for which $L_\x = i$ is equal to
${k\choose i} (v-1)^i$. This completes the proof of the lemma.
\end{proof}

As it can be noted in Figure~\ref{fig:m-2-(3,3)}, in each column
$C_{\x}$ where at least one of the components of $\x$ is 0, i.e.
$L_{\x} \le 2$, there exists at least one row having a pivot 1 in
that column. For example for the column $C_{010}$ both
$(0,1)_{\{1,2\}}$ and $(1,0)_{\{2,3\}}$ rows have such property.
All such columns and one of the rows corresponding to that column
are indicated with a ``\s'' sign and such 1's are shown in bold
face and underlined.

To show the other direction in Theorem~\ref{th:rank}, we prove
the following lemma.

\begin{lemma}\label{lem:20}
For every vector $\x \in V^k$ with $L_\x > t$, we have
\begin{equation}\label{eq:eqn1}
\sum_{\y\in F_\x} (-1)^{L_\y}C_\y = \overline{\bf 0}.
\end{equation}
Where $\overline{{\bf 0}}$ is a vector of appropriate size with
all components equal to $0$.
\end{lemma}

\begin{proof}
It is enough to focus on a fixed row, say $\u_I$, and count the
number of ones in the intersection of this row and columns $C_\y$
for $\y\in F_\x$, by taking the signs in the above expression into
account, and show that the corresponding entry in the left-hand
side of the above equation is 0.

Consider a row $\u_I$ with $I=\{i_1,\ldots,i_t\}$ where
$i_1<\cdots<i_t$.

If there is an element $i_j\in I\setminus A_\x$ such that $u_j\neq
0$, then all the entries in the intersection of this row and
columns $C_\y$ with $\y\in F_\x$ are 0, since for all such $\y$,
$y_{i_j}$ is 0 and therefore is not equal to $u_j$. Thus, the
entry in the left-hand side of Equation~(\ref{eq:eqn1}) in the
row corresponding to $\u_I$ is 0.

Now, suppose $u_j=0$, for every $j$ where $i_j\in I\setminus
A_\x$. Consider the set $Y$ of vectors, $\y\in F_\x$, such that
the entry in the intersection of the row corresponding to $\u_I$
and the column $C_\y$ is 1. For every $\y\in Y$, $\alpha =
|A_\x\cap I|$ of its entries have a value equal to the
corresponding entry in $\u_I$. Therefore, there are $L_\x-\alpha$
entries in $\y$ that can take either a value of 0, or the value
of the corresponding entry in $\x$. We call these $L_\x-\alpha$
entries the {\em free} entries of $\y$. Consider the set of all
$\y\in Y$ that have $j$ non-zero free entries (i.e., are equal to
the corresponding value in $\x$). The number of such $\y$'s is
${L_\x-\alpha\choose j}$ and for each such $\y$, $L_\y=j+\alpha
-\zeta$, where $\zeta$
 is the number of zeros in the intersection of $I$ and $A_x$. Therefore,
the entry in the row corresponding to $\u_I$ in the left-hand side
of Equation~(\ref{eq:eqn1}) is equal to
\[
\sum_{j=0}^{L_\x-\alpha}(-1)^{j+\alpha -\zeta}{L_\x-\alpha\choose
j} =
(-1)^{\alpha-\zeta}\sum_{j=0}^{L_\x-\alpha}(-1)^{j}{L_\x-\alpha\choose
j} = 0.
\]
Thus, all entries of the vector in the left-hand side of
Equation~(\ref{eq:eqn1}) are 0.
\end{proof}
\begin{lemma}\label{lem:21} \quad
${\rm rank}(\M{t}{v}{k})\leq\sum_{i=0}^t{k\choose i} (v-1)^i$.
\end{lemma}
\begin{proof}
By applying Lemma~\ref{lem:20}, repeatedly, to any column  $C_\x$
with  $L_{\x} > t$, we can write $C_\x$ in terms of columns $C_\y$
with $L_{\y}\leq t$. Thus the latter columns form a spanning set
for the column space  of $\M{t}{v}{k}$. We noted earlier that
there are exactly $\sum_{i=0}^t{k\choose i} (v-1)^i$ of such
columns.\end{proof}

Theorem~\ref{th:rank} follows from Lemma~\ref{lem:21} and
Lemma~\ref{lem:19}.
\section{Latin squares and Latin trades}
\label{sec:Latinsquare}

As we noted earlier, a Latin square of order $v$ may be viewed as
an ${\rm OA}_2(v,3,1)$. So the matrix $\M{2}{v}{3}$ is of special
interest. In this section we find a basis for $\M{2}{v}{3}$ which
consists of the so-called intercalates.

We start with a few definitions. A {\sf partial Latin square} $P$
of order $v$ is a $v\times v$ array in which some of the entries
are filled with elements from a set $V=\{0,1,\dots, v-1\}$ in
such a way that each element of $V$ occurs at most once in each
row and at most once in each column of the array. In other words,
there are cells in the array that may be empty, but the positions
that are filled conform with the Latin property of array. Once
again a partial Latin square may be represented as a set of
ordered triples.  However in this case we will include triples of
the form $(i,j;\emptyset)$ and read this to mean that cell
$(i,j)$ of the partial Latin square is empty. The set of cells
${\cal S}_P= \{(i,j)\mid (i,j;P_{ij})\in P, {\rm\ for\ some\ }
P_{ij}\in V\}$ is said to determine the {\sf shape} of $P$ and
$|{\cal S}_P|$ is said to be the {\sf volume} of the partial
Latin square. That is, the volume is the number of nonempty
cells. For each row $r$, $0\leq r\leq v-1$, we let ${\cal R}_P^r$
denote the set of entries occurring in row $r$ of $P$. Formally,
${\cal R}_P^r=\{P_{rj}\mid P_{rj}\in V \wedge (r,j;P_{rj})\in
P\}$. Similarly, for each column $c$, $0\leq c\leq v-1$, we
define ${\cal C}_P^c=\{P_{ic}\mid P_{ic}\in V \wedge
(i,c;P_{ic})\in P\}$.

A {\sf Latin trade}, $T=(P,Q)$, of {\sf volume} $s$ is an ordered
set of two \pls s, of order $v$, such that
\begin{enumerate}
\item ${\cal S}_P={\cal S}_Q,$
\item for each $(i,j)\in {\cal S}_P$, $P_{ij}\neq Q_{ij},$
\item for each $r$, $0\leq r\leq v-1$, ${\cal R}_P^r={\cal R}_Q^r,$ and
\item for each $c$, $0\leq c\leq v-1$, ${\cal C}_P^c={\cal C}_Q^c$.
\end{enumerate}
Thus a Latin trade is a pair of disjoint \pls s of the same shape
and order, which are row-wise and column-wise mutually balanced.
We refer to the shape of a Latin trade $T$ as the shape of the
individual components $P$ and $Q$.

\begin{example}\label{ex:main}
Below is an example of two partial Latin squares which together
form a Latin trade of order 5 and of volume 19. To conserve space
we will display a Latin trade by superimposing one partial Latin
square on top of the other, and using subscripts to differentiate
the entries of the second from those of the first, as shown below.
\end{example}
\begin{figure}[ht]\label{fig:latin trade}
\begin{center}
\begin{tabular}{
 |@{\hspace{4.5pt}}c@{\hspace{3.5pt}}
 |@{\hspace{4.5pt}}c@{\hspace{3.5pt}}
 |@{\hspace{4.5pt}}c@{\hspace{3.5pt}}
 |@{\hspace{4.5pt}}c@{\hspace{3.5pt}}
 |@{\hspace{4.5pt}}c@{\hspace{3.5pt}}
 |}
\hline
\xx. & \xx. & 2 & 3 & 1 \\
\hline
\xx. & 2 & \xx. & 1 & 4 \\
\hline
1 & \xx. & 0 & 4 & 3 \\
\hline
0 & 4 & 1 & \xx. & 2 \\
\hline
4 & 1 & 3 & 2 & 0 \\
\hline
\end{tabular}\quad
\begin{tabular}{
 |@{\hspace{4.5pt}}c@{\hspace{3.5pt}}
 |@{\hspace{4.5pt}}c@{\hspace{3.5pt}}
 |@{\hspace{4.5pt}}c@{\hspace{3.5pt}}
 |@{\hspace{4.5pt}}c@{\hspace{3.5pt}}
 |@{\hspace{4.5pt}}c@{\hspace{3.5pt}}
 |}
\hline
\xx. & \xx. & 1 & 2 & 3 \\
\hline
\xx. & 1 & \xx. & 4 & 2 \\
\hline
4 & \xx. & 3 & 1 & 0 \\
\hline
1 & 2 & 0 & \xx. & 4 \\
\hline
0 & 4 & 2 & 3 & 1 \\
\hline
\end{tabular}\quad\quad\quad\quad
\begin{tabular}{
 |@{\hspace{1pt}}c@{\hspace{1pt}}
 |@{\hspace{1pt}}c@{\hspace{1pt}}
 |@{\hspace{1pt}}c@{\hspace{1pt}}
 |@{\hspace{1pt}}c@{\hspace{1pt}}
 |@{\hspace{1pt}}c@{\hspace{1pt}}
 |}
\hline
\xx. & \xx. & \m21 & \m32 & \m13 \\
\hline
\xx. & \m21 & \xx. & \m14 & \m42 \\
\hline
\m14 & \xx. & \m03 & \m41 & \m30 \\
\hline
\m01 & \m42 & \m10 & \xx. & \m24 \\
\hline
\m40 & \m14 & \m32 & \m23 & \m01 \\
\hline
\end{tabular}
\end{center}
\caption{A Latin trade}
\end{figure}

The concept of a Latin trade in a Latin square is similar to the
concept of a mutually balanced set or a trade in a block design,
see \cite{MR1056530}. The same as trades in design theory, the
discussion of Latin trades is related to intersection problems.
For example, they are relevant to the problem of finding the
possible number of intersections for Latin squares (see \cite{Fu},
\cite{MR1125351}, \cite{MR1278954}, and \cite{MR1874724}).
Also Latin trades arise naturally in the discussion of critical
sets in Latin squares (see for example~\cite{MR1393712} and
\cite{MR2048415}).

Latin trades have been studied by many authors. Fu and
Fu~\cite{MR1125351} used the term ``disjoint and mutually
balanced'' (DMB) partial Latin squares, Keedwell~\cite{MR1393712}
used ``critical partial Latin square'' (CPLS), while Donovan et
al.~\cite{MR1432756} used the term ``Latin interchange''. Adams et
al.~\cite{MR1874724} suggest the terminology `$2$-way Latin
trade'' for consistency with similar concepts in other
combinatorial structures such as block designs, graph colouring,
cycle systems, etc. See for instance~\cite{MR1056530},
\cite{MR1366859}, \cite{MR2001a:05059}, and~\cite{MR1828626} for
further use of trades.

A Latin trade of volume 4 which is unique (up to isomorphism), is
called an {\sf intercalate} (see~Figure~\ref{fig:intercalate}).
\begin{figure}[ht]\label{fig:intercalate}
\begin{center}
\begin{tabular}{
 |@{\hspace{5pt}}c@{\hspace{5pt}}
 |@{\hspace{1pt}}c@{\hspace{1pt}}
 |@{\hspace{5pt}}c@{\hspace{5pt}}
 |@{\hspace{1pt}}c@{\hspace{1pt}}
 |@{\hspace{5pt}}c@{\hspace{5pt}}
 |@{\hspace{5pt}}c@{\hspace{5pt}}
 |}
\hline
\xx. & \xx. &  \xx. &  \xx. &  \xx. & \xx. \\
\hline
\xx. & \m12 & \xx. & \m21 &  \xx. & \xx. \\
\hline
 \xx.& \xx. &  \xx. &  \xx. &  \xx. & \xx. \\
\hline
 \xx.&  \m21 &  \xx. & \m12 &  \xx. & \xx. \\
\hline
 \xx.& \xx. &  \xx. &  \xx. &  \xx. & \xx. \\
\hline
 \xx.& \xx. &  \xx. &  \xx. &  \xx. & \xx. \\
\hline
\end{tabular}
\end{center}
\caption{An intercalate}
\end{figure}

Similar to orthogonal arrays which  correspond to solutions of
the Equation~(\ref{eq:eqnf}), it is clear that any Latin trade
may also be treated as a solution ${\bf T}$, to the equation:
\begin{equation}\label{eq:eqn0} {\rm M}{\bf
T}= \overline{\bf 0},
\end{equation}
where $M = \M{2}{v}{3}$ and ${\bf T}$ is a (signed) frequency
vector derived from the trade $T=(P,Q)$, i.e.,

\ \ ${\bf{T}}_{ijk} =
\begin{cases}
    1        &\mbox{if $(i,j;k) \in P$}\\
    -1       &\mbox{if $(i,j;k) \in Q$} \\
    0       &\mbox{otherwise.}
 \end{cases}$

Therefore Latin trades are in the null space of $\M{2}{v}{3}$.
\begin{theorem}\label{th:basis}
There exists a basis for the null space of the matrix
$\M{2}{v}{3}$ consisting only of intercalates.
\end{theorem}

\begin{proof}
Latin trades are in the null space of $\M{2}{v}{3}$. By
Theorem~\ref{th:rank} we know that
\[
{\rm null}(\M{2}{v}{3}) = (v-1)^3.
\]
For each $i,j,k$; $1 \le i,j,k \le v-1$, consider the $ijk$'th
intercalate defined in Figure~4.

\begin{figure}[ht]\label{4}
\begin{center}
\def\arraystretch{0.7}
\begin{tabular}{
 @{\hspace{1pt}}c@{\hspace{2pt}}|@{\hspace{2pt}}
 @{\hspace{1pt}}c@{\hspace{1pt}}
 @{\hspace{1pt}}c@{\hspace{1pt}}
 @{\hspace{1pt}}c@{\hspace{1pt}}
 @{\hspace{1pt}}c@{\hspace{1pt}}
}
\n{}          &\n{0}       &\n{\cdots}  &\n{j}        &\n{\cdots}\\[0.25em]
\hline\\[-0.5em]
\n{0}         &\m{0}{k}    &\n{\cdots}  &\m{k}{0}     &\n{\cdots}\\
\n{\vdots}    &\n{\vdots}  &\n{\ddots}  &\n{\vdots}   &\n{\ddots}\\
\n{i}         &\m{k}{0}    &\n{\cdots}  &\m{0}{k}     &\n{\cdots}\\
\n{\cdot}     &\n{\cdot}   &\n{\cdots}  &\n{\cdot}    &\n{\cdots}\\
\n{\cdot}     &\n{\cdot}   &\n{\cdots}  &\n{\cdot}    &\n{\cdots}
\end{tabular}
\end{center}
\caption{Basis intercalates}
\end{figure}

There are $(v-1)^3$ of them.
The vectors corresponding to these intercalates are independent,
as for example, the frequency vector of the $ijk$'th intercalate
has an entry $-1$ in the ($i,j,k$) coordinate while all others
have 0 in that coordinate.  Therefore, they form a basis for the
null space of $\M{2}{v}{3}$.
\end{proof}

The above theorem shows, existentially, that every Latin trade can
be written as the sum of intercalates.  In~\cite{MR1880972},
Donovan and Mahmoodian introduced a simple combinatorial
algorithm which enables one to compute such a decomposition. But
by linear algebraic approach and knowing a basis which consists
only of intercalates, makes it straightforward to do this task.
We give an example of this method:

\begin{figure}[ht]\label{fig:intercalate2}
\begin{tabular}{
 |@{\hspace{1pt}}c@{\hspace{1pt}}
 |@{\hspace{1pt}}c@{\hspace{1pt}}
 |@{\hspace{1pt}}c@{\hspace{1pt}}
 |@{\hspace{1pt}}c@{\hspace{1pt}}
 |}
\hline
 \m01 & \m12 &  \m23 &  \m30  \\
\hline
 \m12 & \m21 &  \xx. &  \xx.  \\
\hline
 \m23 & \xx. &  \m32 &  \xx.  \\
\hline
 \m30 & \xx. &  \xx. &  \m03  \\
\hline
\end{tabular}\quad  =  \quad
\begin{tabular}{
 |@{\hspace{1pt}}c@{\hspace{1pt}}
 |@{\hspace{1pt}}c@{\hspace{1pt}}
 |@{\hspace{5pt}}c@{\hspace{5pt}}
 |@{\hspace{5pt}}c@{\hspace{5pt}}
 |}
\hline
 \m01 & \m10 &  \xx. &  \xx.  \\
\hline
 \m10 & \m01 &  \xx. &  \xx.  \\
\hline
 \xx. & \xx. &  \xx. &  \xx.  \\
\hline
 \xx. & \xx. &  \xx. &  \xx.  \\
\hline
\end{tabular}
\quad - \quad
\begin{tabular}{
 |@{\hspace{1pt}}c@{\hspace{1pt}}
 |@{\hspace{1pt}}c@{\hspace{1pt}}
 |@{\hspace{5pt}}c@{\hspace{5pt}}
 |@{\hspace{5pt}}c@{\hspace{5pt}}
 |}
\hline
 \m02 & \m20 &  \xx. &  \xx.  \\
\hline
 \m20 & \m02 &  \xx. &  \xx.  \\
\hline
 \xx. & \xx. &  \xx. &  \xx.  \\
\hline
 \xx. & \xx. &  \xx. &  \xx.  \\
\hline
\end{tabular}
\begin{center}
\hspace*{12mm}+ \quad
\begin{tabular}{
 |@{\hspace{1pt}}c@{\hspace{1pt}}
 |@{\hspace{5pt}}c@{\hspace{5pt}}
 |@{\hspace{1pt}}c@{\hspace{1pt}}
 |@{\hspace{5pt}}c@{\hspace{5pt}}
 |}
\hline
 \m02 & \xx. &  \m20 &  \xx.  \\
\hline
 \xx. & \xx. &  \xx. &  \xx.  \\
\hline
 \m20 & \xx. &  \m02 &  \xx.  \\
\hline
 \xx. & \xx. &  \xx. &  \xx.  \\
\hline
\end{tabular}
\quad - \quad
\begin{tabular}{
 |@{\hspace{1pt}}c@{\hspace{1pt}}
 |@{\hspace{5pt}}c@{\hspace{5pt}}
 |@{\hspace{1pt}}c@{\hspace{1pt}}
 |@{\hspace{5pt}}c@{\hspace{5pt}}
 |}
\hline
 \m03 & \xx. &  \m30 &  \xx.  \\
\hline
 \xx. & \xx. &  \xx. &  \xx.  \\
\hline
 \m30 & \xx. &  \m03 &  \xx.  \\
\hline
 \xx. & \xx. &  \xx. &  \xx.  \\
\hline
\end{tabular}
\quad + \quad
\begin{tabular}{
 |@{\hspace{1pt}}c@{\hspace{1pt}}
 |@{\hspace{5pt}}c@{\hspace{5pt}}
 |@{\hspace{5pt}}c@{\hspace{5pt}}
 |@{\hspace{1pt}}c@{\hspace{1pt}}
 |}
\hline
 \m03 & \xx. &  \xx. &  \m30  \\
\hline
 \xx. & \xx. &  \xx. &  \xx.  \\
\hline
 \xx. & \xx. &  \xx. &  \xx.  \\
\hline
 \m30 & \xx. &  \xx. &  \m03  \\
\hline
\end{tabular}
\end{center}
\caption{An example of the algorithm}
\end{figure}

\section{A basis for the null space of \ $\M{t}{v}{t+1}$}
\label{sec:nullspace}
In this section we generalize the notion of Latin trades and find
a basis for the null space of $\M{t}{v}{t+1}$. First we note that
each Latin trade $T=(P,Q)$ may be represented by a homogeneous
polynomial of order $3$ as follows.  The polynomial is over a
non-commutative ring, hence the terms are ordered
multiplicatively (meaning that $x_{i_1}x_{i_2}x_{i_3}$ is
different from, say $x_{i_2}x_{i_1}x_{i_3}$):
$$
P(x_0,x_1,\ldots, x_{v-1})= \sum_{(i_1, i_2; i_{3}) \in
P}{x_{i_1}x_{i_2}x_{i_3}} - \sum_{(j_1, j_2; j_{3}) \in
Q}{x_{j_1}x_{j_2}x_{j_3}}.
$$
Note that the positive terms correspond to the elements of
 $P$  while the negative terms
correspond to the elements of $Q$. For example the intercalates
which form a basis for the null space of $\M{2}{v}{3}$ in
Theorem~\ref{th:basis} are:
\begin{eqnarray*}
P(x_0,x_1,\ldots, x_{v-1})&=&
x_{0}x_{0}x_{0}+x_{0}x_{j}x_{k}+x_{i}x_{0}x_{k}+x_{i}x_{j}x_{0}\\&&
-\ x_{0}x_{0}x_{k}-x_{0}x_{j}x_{0}-x_{i}x_{0}x_{0}-x_{i}x_{j}x_{k}\\
&=&(x_{0}-x_{i})(x_{0}-x_{j})(x_{0}-x_{k}), \quad 1 \le i,j,k \le
v-1.
\end{eqnarray*}

Similarly, each homogeneous polynomial $ P(x_0,x_1,\ldots,
x_{v-1})$ of order $k$, whose terms are ordered multiplicatively,
corresponds to a frequency vector ${\bf T}$. If ${\bf T}$
satisfies
\begin{equation}\label{eq:eqn4} {\rm M}{\bf
T}= \overline{\bf 0},
\end{equation}
where $M = \M{t}{v}{k}$, we call it a $\SFT{t}{v}{k}$. So, a
 Latin trade defined in
Section~\ref{sec:Latinsquare} is also $\T{2}{v}{3}$. Any
$\T{t}{v}{t+1}$ of the following form will be called a
$\SFI{t}{v}{t+1}$:
$$
P(x_0,x_1,\ldots, x_{v-1})=
(x_{i_1}-x_{j_1})(x_{i_2}-x_{j_2})\cdots(x_{i_{t+1}}-x_{j_{t+1}}),$$
where ${i_{m}} \ {\rm and} \ {j_{n}}\in \{0, \ldots, v-1\}$,  and
for each $l$, $i_l$ is distinct from $j_l$.
A $\I{2}{v}{3}$ is simply an intercalate in a Latin square,
defined in the previous section.
\begin{theorem}\label{th:basistt+1}
There exists a basis for the null space of the matrix
$\M{t}{v}{t+1}$ consisting only of $\ITI{t}{v}{t+1}$s.
\end{theorem}

\begin{proof}
It can easily be checked that $\I{t}{v}{t+1}$s are included in
the null space of $\M{t}{v}{t+1}$. By Theorem~\ref{th:rank} we
know that \  $ {\rm null}(\M{t}{v}{t+1}) = (v-1)^{t+1}. $

Consider the following set of  $(v-1)^{t+1}$  intercalates:
\begin{eqnarray*}
&&P(x_0,x_1,\ldots, x_{v-1})=
(x_{0}-x_{i_1})(x_{0}-x_{i_2})\cdots(x_{0}-x_{i_{t+1}}),
\\&&\hspace*{7cm} 1 \le {i_{1}},{i_{2}},\ldots,{i_{k}}\le v-1.
\end{eqnarray*}

The intercalates in this set are independent.  For example, the
frequency vector of the ${i_{1}}{i_{2}}\cdots{i_{t+1}}$'th
intercalate has a non-zero entry, namely $(-1)^{t+1}$, in the
(${i_{1}},{i_{2}},\ldots,{i_{t+1}}$)th coordinate, while all
others have 0 in that coordinate.  Therefore, they form a basis
for the null space of $\M{t}{v}{t+1}$.
\end{proof}

\vspace*{4mm}
Finally, we note that given the above generalization of the
concept of Latin trades, many questions similar to the ones in
the theory of $t$-trades in block designs, may be raised. For
example, it would be interesting to characterize the possible
support sizes of $\T{t}{v}{k}$s.
\section*{Acknowledgement}
The authors appreciate comments of John van Rees, specially for
the last lines of the proof of Lemma~\ref{lem:20}. This work was
partly done while the third author was spending his sabbatical
leave in the Microsoft and also in the Institute for Advanced
Studies in Basic Sciences (IASBS), Zanjan and finally in the
Institute for Studies in Theoretical Physics and
 Mathematics (IPM). . He would like to thank
all these institutions for their warm and generous hospitality
and support.
\def\cprime{$'$}

\end{document}